\documentclass[12pt,reqno]{amsart}

\usepackage{amssymb}

\newtheorem{theorem}{Theorem}[section]
\newtheorem{proposition}[theorem]{Proposition}
\newtheorem{lemma}[theorem]{Lemma}

\newtheorem{definition}[theorem]{Definition}
\newtheorem{remark}[theorem]{Remark}

\numberwithin{equation}{section}

\begin{document}
\baselineskip=15.5pt

\title[Orbifold projective structures and symplectic form]{On the symplectic
structure over a moduli space of orbifold projective structures}

\author[P. Ar\'es-Gastesi]{Pablo Ar\'es-Gastesi}

\address{Department of Applied Mathematics and Statistics, School of Economics
and Business, Universidad CEU San Pablo, Madrid, Spain}

\email{pablo.aresgastesi@ceu.es}

\author[I. Biswas]{Indranil Biswas}

\address{School of Mathematics, Tata Institute of Fundamental
Research, Homi Bhabha Road, Bombay 400005, India}

\email{indranil@math.tifr.res.in}



\date{}

\begin{abstract}
Let $S$ be a compact connected oriented smooth orbifold surface. We show that
using Bers simultaneous uniformization, the moduli space of
projective structures on $S$ can
be mapped biholomorphically onto the
total space of the holomorphic cotangent bundle of the Teichm\"uller space
for $S$. The total space of the holomorphic cotangent bundle of the Teichm\"uller
space is equipped with the Liouville holomorphic symplectic form, and the
moduli space of projective structures also has a natural holomorphic symplectic
form. The above identification between the
moduli space of projective structures on $S$ and the holomorphic cotangent bundle
of the Teichm\"uller space for $S$
is proved to be compatible with these
symplectic structures. Similar results are obtained for biholomorphisms
constructed using uniformizations provided by Schottky groups and Earle's version of
simultaneous uniformization.
\end{abstract}

\maketitle

\section{Introduction}\label{sec1}

The holomorphic automorphisms of the complex projective line ${\mathbb C}
{\mathbb P}^1$ are of the form $z\, \longmapsto\, (az+b)/(cz+d)$, where
$a,\, b,\, c,\, d$ are complex numbers with $ad-bc\,=\, 1$; these are known
as M\"obius transformations.
A projective structure on a $C^\infty$ compact oriented surface $R$ is defined
by a covering of $R$ by coordinate charts, compatible with the
orientation, so that all the transition functions
are M\"obius transformations. Two projective structures on $R$
are considered isomorphic if they differ by a diffeomorphism of $R$
homotopic to the identity map of $R$. Let $\mathcal P(R)$ denote the space
of all isomorphism classes of projective structures on $R$.

Consider the space of all complex structures on $R$ compatible with the orientation.
Two of them are called isomorphic if they differ by a diffeomorphism of $R$
homotopic to the identity map of $R$. Let $\mathcal T(R)$ denote the Teichm\"uller
space of $R$ that parametrizes the isomorphism classes of complex structures on $R$.
Clearly, a projective structure on $R$ induces a complex structure
on $R$ compatible with the orientation. So there is a natural map
$$\varphi\,:\,\mathcal P(R) \,\longrightarrow\, \mathcal T(R)\, .$$ Both
$\mathcal P(R)$ and $\mathcal T(R)$ are equipped with complex structures, and
the map $\varphi$ is holomorphic.

It is well-known that the space of all projective structures on a fixed Riemann
surface can be identified with the space of quadratic differentials on that
surface (see \cite[p. 292]{Nag}). This means that any $C^\infty$ section
\[
f\,:\, \mathcal T(R) \,\longrightarrow\, \mathcal P(R)
\]
of the above projection $\varphi$ produces a diffeomorphism
\[
T_f\,:\,T^*\mathcal T(R) \,\longrightarrow\, \mathcal P(R)\, ,
\]
where $T^*\mathcal T(R)$ is the holomorphic cotangent bundle of the Teichm\"uller
space (its fibers are identified with the space of quadratic differentials).
The above diffeomorphism $T_f$ is holomorphic if and only if $f$ is holomorphic.

Both $T^*\mathcal T(R)$ and $\mathcal P(R)$ have natural holomorphic
symplectic structures. Since $T^*\mathcal T(R)$ is a cotangent bundle, it
has the Liouville symplectic form
$$
\Omega_{\mathcal T}\,:=\, d\sigma\, ,
$$
where $\sigma$ is the tautological holomorphic one-form on $T^*\mathcal T(R)$. On
the other hand, any projective structure on $R$ produces
a flat principal $\text{PSL}(2,{\mathbb C})$--bundle
on $R$ (recall that the transition functions for a projective structure lie
in $\text{PSL}(2,{\mathbb C})$). Now taking monodromy of flat connections,
the space $\mathcal P(R)$ is mapped to an open subset of the smooth part of
the representation space
$$
\text{Hom}(\pi_1(R)\, , \text{PSL}(2,{\mathbb C}))/\text{PSL}(2,{\mathbb C})\, .
$$
This map is a local biholomorphism.
The smooth part of $$\text{Hom}(\pi_1(R)\, , \text{PSL}(2,{\mathbb C}))/\text{PSL}(2,
{\mathbb C})$$ is equipped with a holomorphic symplectic form
\cite{AtBo}, \cite{Go}. Pulling back this $\mathcal P(R)$ we get a holomorphic
symplectic form $\Omega_{\mathcal P}$ on $\mathcal P(R)$. In \cite{ka}, Kawai
showed that if $f$ is Bers' section $B$, then one has
\[
T^*_B\Omega_{\mathcal P} \,=\, \pi\cdot \Omega_{\mathcal T}\, .
\]
This result was extended to the Schottky's and
Earle's sections in \cite{Bis} and \cite{AB} respectively.

An orbifold surface is a surfaces with weighted marked points.
Our aim here is to address the question whether the above set-up generalizes to
orbifolds, and whether similar results hold for orbifolds. We answer these
questions affirmatively. More concretely, we
begin by recalling the definition of an orbifold surface $S$, and
explaining what a projective structure on an orbifold means. This leads us to
the definitions of Teichm\"uller space $\mathcal T(S)$ and the space
of projective structures $\mathcal P(S)$ for $S$. As in the surface
case, there is a natural holomorphic projection
\[
\widetilde{f}_S\, :\, {\mathcal P}(S)\,\longrightarrow\, {\mathcal T}(S)
\]
that sends a projective structure to its underlying complex structure.
There is a natural holomorphic symplectic structure on ${\mathcal P}(S)$,
which we will denote by $\Omega^S_{\mathcal P}$.

By a Galois covering of a surface we will mean a covering map of it which is 
possibly ramified (locally isomorphic to $z\, \longmapsto\, z^n$ for some positive
integer $n$) such that the group of deck transformations acts transitively on
every fiber of the covering map.

To define a section of the above projection $\widetilde{f}_S$, we consider Bers'
section $B$ of an appropriate finite Galois cover of $S$. Then we average
$B$ over the Galois group $\Gamma$. This construction gives us a biholomorphism
\[
T_{S,B}\,:\, T^*\mathcal T(S) \,\longrightarrow\, \mathcal P(S)\, .
\]
Let $\Omega^S_{\mathcal T}$ denote the Liouville symplectic form on
$T^*\mathcal T(S)$. In Theorem 4.2 we show that this mapping $T_{S,B}$ preserves
the symplectic structures of the spaces, up to a constant:
\[
T^*_{S,B}\Omega^S_{\mathcal P}\,=\, \pi\cdot
\Omega^S_{\mathcal T}\, .
\]
Finally, we generalize this result to the biholomorphisms $T^*\mathcal T(S) \,
\longrightarrow\, \mathcal P(S)$ corresponding to Schottky's and Earle's sections.

\section{Orbifold Riemann surface and projective structure}

\subsection{Definition of orbifold projective structure}

Let ${\mathbb N}^{> 1}$ be the set of integers bigger than
one. An \textit{orbifold surface} is a triple
$(X\, ,{\mathbb D}\, , \varpi)$, where $X$ is a
compact connected oriented $C^\infty$ surface,
$$
{\mathbb D} \,:=\,\{x_1\, ,\cdots \, , x_n\}\,\subset\, X
$$
is a finite collection of distinct ordered points, and
\begin{equation}\label{vp}
\varpi\, :\, {\mathbb D} \,\longrightarrow \, {\mathbb N}^{> 1}
\end{equation}
is a function. Since the elements of ${\mathbb D}$ are ordered, $\varpi$
can be considered as a function on $\{1\, ,\cdots\, ,n\}$.

A \textit{coordinate function} on
$(X\, ,{\mathbb D}\, , \varpi)$
is a pair of the form $(V\, ,\phi)$, where
$V\,\subset {\mathbb C}{\mathbb P}^1$ is a connected open
subset, and
\begin{equation}\label{e1}
\phi\, :\, V\, \longrightarrow\, X
\end{equation}
is an orientation preserving $C^\infty$ open map, such that
$\# \phi(V)\bigcap {\mathbb D}\,\leq \, 1$, and
\begin{enumerate}
\item if $\phi(V)\bigcap {\mathbb D}\,=\,
\emptyset$, then $\phi$ is an embedding, and

\item if $\phi(V)\bigcap {\mathbb D}\, =\, x_i$,
then $\phi$ is a ramified Galois covering of $\phi(V)$ with Galois group
${\mathbb Z}/\varpi(x_i) {\mathbb Z}$, and it is totally ramified over $x_i$
but unramified over the complement $\phi(V)\setminus \{x_i\}$.
\end{enumerate}
The second condition implies that $\phi(V)$ can contain
at most one point of $\mathbb D$.

The group $\text{SL}(2, {\mathbb C})$ acts on ${\mathbb C}{\mathbb 
P}^1$; the action of any
$$
\begin{pmatrix}
a & b\\
c & d
\end{pmatrix} \, \in \, \text{SL}(2, {\mathbb C})
$$
sends any $z\,\in\, {\mathbb C}{\mathbb P}^1\,=\,
{\mathbb C}\cup\{\infty\}$ to $(az+b)/(cz+d)\,\in\, {\mathbb C}{\mathbb P}^1$. This
action of $\text{SL}(2,
{\mathbb C})$ factors through the quotient group $\text{PGL}(2, {\mathbb C})\,=\,
\text{SL}(2, {\mathbb C})/\pm I$. This way,
$\text{PGL}(2, {\mathbb C})$ gets identified with the
group of all holomorphic automorphisms of
${\mathbb C}{\mathbb P}^1$. The holomorphic automorphisms of
${\mathbb C}{\mathbb P}^1$ are also called M\"obius transformations.

A \textit{projective atlas} on $(X\, ,{\mathbb D}\, , \varpi)$
is a collection of coordinate functions
$\{(V_j\, ,\phi_j)\}_{j\in J}$ such that
\begin{enumerate}
\item $\bigcup_{j\in J} \phi_j(V_j)\, =\, X$,

\item for every $j\,\in\, J$, each deck transformation of the
Galois covering
$$
\phi_j\, :\, V_j\, \longrightarrow \phi_j(V_j)
$$
is the restriction of some M\"obius transformation (if
$\phi_j$ is an embedding, then this condition is automatically
satisfied because the Galois group is trivial), and

\item for every $j\, ,k\,\in\, J$, and for every
connected and simply connected open subset
$$
V\, \subset\,
\phi^{-1}_k((U_j\bigcap U_k)\setminus {\mathbb D})\, ,
$$
each branch of the function $\phi^{-1}_j\circ\phi_k$ over
$V$ is the restriction of some M\"obius transformation.
\end{enumerate}

By a branch of $\phi^{-1}_j\circ\phi_k$ over $V$ we mean a holomorphic
map $$f\, :\, V\, \longrightarrow\, {\mathbb C}{\mathbb P}^1$$
such that $\phi_k \,=\, \phi_j\circ f$. Note that if
$f\, :\, V\, \longrightarrow\, {\mathbb C}{\mathbb P}^1$ is continuous
and $$\phi_k \,=\, \phi_j\circ f\, ,$$ then $f$ is holomorphic.

In view of the second condition in the above definition of
projective structure, if some branch of the function 
$\phi^{-1}_j\circ\phi_k$ over
$V$ is the restriction of some M\"obius transformation, then
each branch of the function $\phi^{-1}_j\circ\phi_k$ over
$V$ is the restriction of some M\"obius transformation.

Two projective atlases $\{(V_j\, ,\phi_j)\}_{j\in J}$ and $\{(V_i\, ,\phi_i)
\}_{i\in I}$ will be called \textit{equivalent} if their union $\{(V_j\, ,
\phi_j)\}_{j\in J\cup I}$ is also a projective atlas.

\begin{definition}\label{def1}
{\rm A} projective structure {\rm on $(X\, ,{\mathbb D}\, , \varpi)$
is an equivalence class of projective atlases.}

{\rm Given a projective structure $P$ on $(X\, ,{\mathbb D}\, , 
\varpi)$, a coordinate function $(V\, ,\phi)$ is called}
compatible with $P$ {\rm if $(V\, ,\phi)$ lies in some
projective atlas in the equivalence class defined by $P$.}
\end{definition}

When the orbifold structure $({\mathbb D}\, , \varpi)$
on $X$ is clear from the context, a projective structure on
$(X\, ,{\mathbb D}\, , \varpi)$ will also
be called an \textit{orbifold projective structure} on $X$.

A projective structure on $(X\, ,{\mathbb D}\, , \varpi)$ produces
a complex structure on $X$. Indeed, this is an immediate consequence
of the following fact: if 
$$
{\mathbb C}{\mathbb P}^1 \,\supset \, V\, 
\stackrel{\phi}{\longrightarrow}\, X
$$
is a coordinate map (as in \eqref{e1}) such
that each deck transformation of $\phi$
is the restriction of some M\"obius transformation, then there is
a unique complex structure on $\phi(V)$ such that $\phi$ is
a holomorphic map.

An \textit{orbifold Riemann surface} is an orbifold surface
$(X\, ,{\mathbb D}\, , \varpi)$ such that $X$ is equipped with
a complex structure compatible with the orientation of $X$.

{}From the above observation that a projective structure produces
a complex structure it follows that a projective
structure on an orbifold surface produces an
orbifold Riemann surface.

Given an orbifold Riemann surface
$(X\, ,{\mathbb D}\, , \varpi)$, a projective structure $P$
on the orbifold surface $(X\, ,{\mathbb D}\, , \varpi)$ will be
called \textit{compatible} with the complex structure if the
complex structure on $X$ given by $P$ coincides with the given
complex structure on $X$. A compatible projective structure on the
orbifold Riemann surface
$(X\, ,{\mathbb D}\, , \varpi)$ will also be called a
\textit{projective structure} on the orbifold Riemann surface
$(X\, ,{\mathbb D}\, , \varpi)$.

When the orbifold structure $({\mathbb D}\, , \varpi)$ on the
Riemann surface $X$ is clear from the context, a projective
structure on the orbifold Riemann surface $(X\, ,{\mathbb D}\, , 
\varpi)$ will also be called an \textit{orbifold projective
structure on the Riemann surface $X$}.

We now recall Lemma 3.2 of \cite{Bi} on the existence of orbifold 
projective structures.

\begin{lemma}[\cite{Bi}]\label{lem1}
An orbifold Riemann surface $(X\, ,{\mathbb D}\, , \varpi)$
admits a compatible projective structure if and only if at least
one of the following three conditions are satisfied:
\begin{enumerate}
\item{} ${\rm genus}(X)\, \geq\, 1$,

\item $\# {\mathbb D} \, \notin \, \{1\, ,2\}$,

\item if $\# {\mathbb D} \,=\, 2$, then
$\varpi(x_1)\,=\,\varpi(x_2)$.
\end{enumerate}
Therefore, $(X\, ,{\mathbb D}\, , \varpi)$ does not
admit a compatible projective structure if and only if
either ${\rm genus}(X)\,=\, 0\, =\, n-1$ or
${\rm genus}(X)\,=\, 0\, =\, n-2$ with
$\varpi(x_1)\,\not=\,\varpi(x_2)$.
\end{lemma}

\noindent
\textbf{Assumption A:}\, Henceforth, for all orbifold surfaces
considered, we assume that at least
one of the three conditions in Lemma \ref{lem1} is satisfied.

In view of Assumption A and Lemma \ref{lem1},
all orbifold Riemann surfaces considered henceforth admit
a projective structure.

\subsection{Parameter space for orbifold projective structures}

Let $S$ be a compact connected oriented surface of genus $g$. Fix
$n$ ordered points ${\mathbb D}\, :=\, \{x_1\, , \cdots\, , x_n\}$ on $S$.
Let ${\mathcal T}(S)$ be the Teichm\"uller space corresponding to
this $n$-pointed surface $S$. We recall a construction of ${\mathcal T}(S)$.
The space of all complex structures
on the smooth surface $S$ compatible with its orientation
will be denoted by ${\rm Com}(S)$. Let $\text{Diff}_{\mathbb D}(S)$ be the
group of all orientation preserving diffeomorphisms of $S$ that fix
the subset $\{x_1\, , \cdots\, , x_n\}$ pointwise. Let
$$
\text{Diff}^0_{\mathbb D}(S)\,\subset\, \text{Diff}_{\mathbb D}(S)
$$
be the subgroup consisting of all diffeomorphisms homotopic, fixing $\mathbb D$
pointwise, to the identity map of $S$. This group $\text{Diff}^0_{\mathbb D}(S)$
acts on ${\rm Com}(S)$; the action of any $f\, \in\, \text{Diff}^0(S)$ sends a
complex structure to its pullback using $f^{-1}$. The above Teichm\"uller space
${\mathcal T}(S)$ is the quotient
$$
{\mathcal T}(S)\, =\, {\rm Com}(S)/\text{Diff}^0_{\mathbb D}(S)\, .
$$
The space ${\rm Com}(S)$ has a natural complex structure which
induces a complex structure on ${\mathcal T}(S)$.

Let $\text{Proj}(S)$ denote the space of all projective structures
on $(S\, ,{\mathbb D}\, , \varpi)$. Consider the group of diffeomorphisms
$\text{Diff}^0_{\mathbb D}(S)$ defined above. This group has a
natural action on $\text{Proj}(S)$. The action of any $\tau \, \in\, 
\text{Diff}^0_{\mathbb D}(S)$ on $\text{Proj}(S)$ takes a projective 
structure $P$ to the one uniquely
determined by the following property: a coordinate function $(V\, 
,\phi)$ is compatible with this projective structure if and only if the
coordinate function $(V\, ,\tau^{-1}\circ\phi)$ is compatible with $P$.
Let
\begin{equation}\label{e2}
{\mathcal P}(S) \,:=\, \text{Proj}(S)/\text{Diff}^0_{\mathbb D}(S)
\end{equation}
be the quotient. There is a natural projection
\begin{equation}\label{e4}
\widetilde{f}_S\, :\, {\mathcal P}(S)\,\longrightarrow\, {\mathcal T}(S)
\end{equation}
that sends a projective structure to the complex structure underlying it.

There is a natural complex structure on $\text{Proj}(S)$ that induces a
complex structure on the quotient space ${\mathcal P}(S)$. An alternative way
of describing this complex structure on $\text{Proj}(S)$ is the following.
A projective structure on the orbifold $S$ defines a flat principal
$\text{PSL}(2, {\mathbb C})$--bundle on the complement $S\setminus \mathbb D$. Sending
a flat connection to its monodromy representation, the space
${\mathcal P}(S)$ gets identified with a submanifold of the smooth locus of
the representation space
$$\text{Hom}(\pi_1(S\setminus {\mathbb D}),\, \text{PSL}(2,
{\mathbb C}))/\text{PSL}(2, {\mathbb C})\, .$$ The smooth locus of
$\text{Hom}(\pi_1(S\setminus {\mathbb D}),\, \text{PSL}(2,
{\mathbb C}))/\text{PSL}(2, {\mathbb C})$ has a complex structure given by the
complex structure on $\text{PSL}(2, {\mathbb C})$, and the submanifold
${\mathcal P}(S)$ is preserved by the underlying almost complex structure.
Therefore, ${\mathcal P}(S)$ gets an induced complex structure. The projection
$\widetilde{f}_S$ in \eqref{e4} is holomorphic.

\begin{proposition}\label{prop1}
As before, let $g$ denote the genus of $S$.
The dimension of this complex manifold ${\mathcal P}(S)$ is
\begin{itemize}
\item $6g-6+2n$ 
if ${\rm genus}(S)\,\geq \, 2$,

\item $2n$ (respectively,
$2$) if ${\rm genus}(S)\, =\, 1$ with $n\,
>\, 0$ (respectively, $n\,=\,0$), and

\item $2(n-3)$ (respectively, $0$) if ${\rm genus}(S)\, =\, 0$ with
$n\, \geq\, 4$ (respectively, $n\,\leq\,3$).
\end{itemize}
\end{proposition}

\begin{proof}
Since the two cases ${\rm genus}(S)\,=\, 0\, =\, n-1$ and
${\rm genus}(X)\,=\, 0\, =\, n-2$ with
$\varpi(x_1)\,\not=\,\varpi(x_2)$ are omitted (see Assumption A),
a theorem due to Bundgaard--Nielsen and Fox says that there is
a finite Galois covering
\begin{equation}\label{e3}
\psi\, :\, Y\, \longrightarrow\, S
\end{equation}
such that $\psi$ is unramified over $S\setminus\mathbb D$,
and for each $x_i\, \in\, {\mathbb D}$, the order of ramification
at every point of $\psi^{-1}(x_i)$ is $\varpi(x_i)$
\cite[p. 26, Proposition 1.2.12]{Na}, where $\varpi$ is the function
in \eqref{vp}. We call
the order of ramification at $0$ of the map $z\, \longmapsto\,
z^m$ to be $m$. Let $\widetilde g$ denote the genus of $Y$.

Let ${\rm Proj}_0(Y)$ denote the space of all projective structures
on the compact oriented surface $Y$ (for $Y$, the subset of
orbifold points is empty); the subscript ``$0$'' is to emphasize that
the orbifold structure on $Y$ is trivial. The 
space of all complex structures
on the smooth surface $Y$ compatible with its orientation
will be denoted by ${\mathcal C}(Y)$. There is a natural map
\begin{equation}\label{fpy}
f'_Y\, :\, {\rm Proj}_0(Y)\, \longrightarrow\, {\mathcal C}(Y)
\end{equation}
that sends a projective structure on $Y$ to the complex structure
on $Y$ defined by it.

Let ${\rm Diff}^0(Y)$ denote the
group of all diffeomorphisms of $Y$ homotopic to the identity
map of $Y$. The group ${\rm Diff}^0(Y)$ acts on both
${\rm Proj}(Y)$ and ${\mathcal C}(Y)$. Define
$$
{\mathcal P}(Y)\, :=\, {\rm Proj}_0(Y)/{\rm Diff}^0(Y)
~\,~\,\text{~and~}\,~\,~ {\mathcal T}(Y)\, :=\, {\mathcal C}(Y)
/{\rm Diff}^0(Y)\, .
$$
The quotient ${\mathcal T}(Y)$ is called the \textit{Teichm\"uller
space} for $Y$. It is a complex manifold of dimension $3{\widetilde g}-3$
or $1$ or $0$ depending on whether ${\widetilde g}\, \geq\, 2$ or ${\widetilde
g}\,=\, 1$ or ${\widetilde g}\,=\, 0$. Also, ${\mathcal T}(Y)$ is contractible
(diffeomorphic to the unit ball). The quotient ${\mathcal P}(Y)$ is a complex 
manifold of 
dimension $2\cdot\dim_{\mathbb C} {\mathcal T}(Y)$, and it is
also contractible. The map $f'_Y$ in
\eqref{fpy} descends to a projection
\begin{equation}\label{fy}
f_Y\, :\, {\mathcal P}(Y)\, \longrightarrow\, {\mathcal T}(Y)\, .
\end{equation}
This map $f_Y$ is a holomorphic submersion. More precisely, $f_Y$ makes
${\mathcal P}(Y)$ a holomorphic fiber bundle over
${\mathcal T}(Y)$. In fact, ${\mathcal P}(Y)$ is a torsor
for the holomorphic cotangent bundle $T^*{\mathcal T}(Y)$,
which means that for any $Z\, \in\, {\mathcal T}(Y)$
the vector space of $T^*_Z {\mathcal T}(Y)$ acts
freely and transitively on the fiber of $f_Y$ over the point $Z$. That
${\mathcal P}(Y)$ is a torsor for $T^*{\mathcal T}(Y)$ follows from the
facts that the space of all projective structure on
a given compact Riemann surface $Z$ is an affine space for
the space of quadratic differentials
$H^0(Z,\, T^*Z\otimes T^*Z)$, while the fiber of
$T^*{\mathcal T}(Y)$ at any point $Z\, \in\, {\mathcal T}(Y)$
is also $H^0(Z,\, T^*Z\otimes T^*Z)$.

Let
\begin{equation}\label{Ga}
\Gamma\, :=\, \text{Gal}(\psi)
\end{equation}
be the Galois group for the covering map $\psi$ in \eqref{e3}.
We will show that $\Gamma$ has a natural action on both
${\mathcal P}(Y)$ and ${\mathcal T}(Y)$.

Take any $T\,\in\, \Gamma$. Since $T$ is an orientation preserving
diffeomorphism of $Y$, it produces a self--map of $\text{Proj}_0(Y)$
that sends a projective structure $P$ to the one uniquely
determined by the following property: a coordinate function $(V\, 
,\phi)$ is compatible with this projective structure if and only if the
coordinate function $(V\, ,T^{-1}\circ\phi)$ is compatible with $P$.
This way we get an action of $\Gamma$ on $\text{Proj}_0(Y)$.
Similarly, $\Gamma$ acts on the space of complex structures ${\mathcal
C}(Y)$: the action of any $T\,\in\, \Gamma$ sends a complex structure
to the pullback of it by $T^{-1}$. The map $f'_Y$ in \eqref{fpy}
evidently intertwines these actions of $\Gamma$ on $\text{Proj}_0(Y)$
and ${\mathcal C}(Y)$.

Next we note that the conjugation action
$\Gamma$ on the group of diffeomorphisms of $Y$ preserves
${\rm Diff}^0(Y)$, meaning for any $T\, \in\, \Gamma$ and
$T'\,\in\,{\rm Diff}^0(Y)$, we have
$$
T^{-1}T'T \,\in\, {\rm Diff}^0(Y)\, .
$$
Therefore, the above actions of $\Gamma$ on
$\text{Proj}_0(Y)$ and ${\mathcal C}(Y)$ descend to actions of
$\Gamma$ on the quotient spaces ${\mathcal P}(Y)$ and ${\mathcal T}(Y)$
respectively. Since $f'_Y$ in \eqref{fpy} is $\Gamma$--equivariant, the
descended map $f_Y$ in \eqref{fy} is also $\Gamma$--equivariant.

{}From the construction of the actions of $\Gamma$ on
${\mathcal P}(Y)$ and ${\mathcal T}(Y)$ it follows that these
actions preserve the complex structures of
${\mathcal P}(Y)$ and ${\mathcal T}(Y)$.

The space ${\mathcal P}(S)$ in \eqref{e2} is the fixed point
locus
\begin{equation}\label{fix}
{\mathcal P}(S)\,=\, {\mathcal P}(Y)^\Gamma\,\subset\,{\mathcal P}(Y)
\end{equation}
for the action of $\Gamma$ on ${\mathcal P}(Y)$.
Indeed, the pullback to $Y$ of a projective structure on
$(S\, ,{\mathbb D}\, , \varpi)$ is a projective structure on
$Y$. This pulled back projective structure on $Y$ is
clearly invariant under the action of $\Gamma$. Conversely,
if we have a $\Gamma$--invariant projective structure $P$
on $Y$, then $P$ defines a projective structure $P'$ on
$(S\, ,{\mathbb D}\, , \varpi)$. A coordinate function
$\phi\, :\, V\, \longrightarrow\, S$, with $V$ simply connected, is
compatible with $P'$ if the lift $V\, \longrightarrow\, Y$ of $\phi$
is compatible with $P$; note that in view of the definition of a projective
structure on $(S\, ,{\mathbb D}\, , \varpi)$, the properties of the covering
map $\psi$ imply that $\phi$ lifts to a map to $Y$.

Since the action of $\Gamma$ on ${\mathcal P}(Y)$ preserves
the complex structure of ${\mathcal P}(Y)$, the fixed
point locus ${\mathcal P}(Y)^\Gamma$ is a complex submanifold
of ${\mathcal P}(Y)$.

To compute the dimension of the fixed point locus ${\mathcal P}(S)$,
we first note that the map $f_Y$ being $\Gamma$--equivariant
restricts to a map
\begin{equation}\label{Fy}
F_Y\,:=\, f_Y\vert_{{\mathcal P}(Y)^\Gamma}\, :\,{\mathcal P}(S)
\,=\, {\mathcal P}(Y)^\Gamma\,\longrightarrow\,
{\mathcal T}(Y)^\Gamma\, ,
\end{equation}
where ${\mathcal T}(Y)^\Gamma\,\subset\, {\mathcal T}(Y)$
is the fixed point locus for the action of $\Gamma$ on
${\mathcal T}(Y)$. We note that ${\mathcal T}(Y)^\Gamma$
is a complex submanifold because the action of $\Gamma$ preserves
the complex structure of ${\mathcal T}(Y)$.
Any $\Gamma$--invariant complex structure on $Y$ produces a complex structure
on $S$. On the other hand, any complex structure on $S$ defines a
$\Gamma$--invariant complex structure on $Y$. It is known that
${\mathcal T}(Y)^\Gamma$ is identified with the earlier defined Teichm\"uller
space ${\mathcal T}(S)$ for the $n$-pointed surface $S$ \cite{Har}.
In particular, $\dim {\mathcal T}(Y)^\Gamma$ coincides with the dimension of
${\mathcal T}(S)$. Therefore,
\begin{itemize}
\item if $\text{genus}(S)\, \geq\, 2$, then $\dim {\mathcal T}(Y)^\Gamma\,=\,
3g-3+n$,

\item if $\text{genus}(S)\, =\, 1$ and $n\,\geq \,1$, then $\dim {\mathcal T}
(Y)^\Gamma\,=\, n$,

\item if $\text{genus}(S)\, =\, 1$ and $n\,= \,0$, then $\dim {\mathcal T}
(Y)^\Gamma\,=\, 1$,

\item if $\text{genus}(S)\, =\, 0$ and $n\,\geq \,4$, then $\dim {\mathcal T}
(Y)^\Gamma\,=\, n-3$, and

\item if $\text{genus}(S)\, =\, 0$ and $n\,\leq \,3$, then $\dim {\mathcal T}
(Y)^\Gamma\,=\, 0$.
\end{itemize}

The map $F_Y$ in \eqref{Fy} is a holomorphic fiber bundle whose fiber
over any Riemann surface $Z\,\in\, {\mathcal T}(Y)^\Gamma$ is an affine
space for the complex vector space consisting of all $\Gamma$--invariant
holomorphic sections of $H^0(Z,\, K^{\otimes 2}_Z)$ (the space of all
$\Gamma$--invariant quadratic differentials on $Z$). Indeed, this follows from
the fact that the space of all $\Gamma$--invariant projective structures on
$Z$ is an affine space for $H^0(Z,\, K^{\otimes 2}_Z)^\Gamma$.

Take any point $$Z\,\in\, {\mathcal T}(Y)^\Gamma\, .$$ We consider
$Z$ as $Y$ equipped with a $\Gamma$--invariant complex structure. We have
$$
H^0(Z,\, K^{\otimes 2}_Z)^\Gamma\,=\, H^0(Z/\Gamma,\, K^{\otimes 2}_{Z/\Gamma}
\otimes {\mathcal O}_{Z/\Gamma}({\mathbb D}))\, .
$$
Using Serre duality, we have
$$
H^0(Z/\Gamma,\, K^{\otimes 2}_{Z/\Gamma}        
\otimes {\mathcal O}_{Z/\Gamma}({\mathbb D}))\,=\,
H^1(Z/\Gamma,\, T({Z/\Gamma})
\otimes {\mathcal O}_{Z/\Gamma}(-{\mathbb D}))^*\, ,
$$
where $T({Z/\Gamma})$ is the holomorphic tangent bundle of $Z/\Gamma$.
But $$H^1(Z/\Gamma,\, T({Z/\Gamma})
\otimes {\mathcal O}_{Z/\Gamma}(-{\mathbb D}))$$ is the holomorphic tangent
space to ${\mathcal T} (Y)^\Gamma$ at the point $Z/\Gamma\,\in\,
{\mathcal T} (Y)^\Gamma$. Hence
$\dim {\mathcal T} (Y)^\Gamma\,=\, \dim H^0(Z,\, K^{\otimes 2}_Z)^\Gamma$. Since
$$\dim {\mathcal P}(S)\,=\, \dim {\mathcal T} (Y)^\Gamma+
\dim H^0(Z,\, K^{\otimes 2}_Z)^\Gamma\, ,
$$
the proposition follows.
\end{proof}

\begin{remark}\label{rem1}
Take any $Z/\Gamma\, \in\, {\mathcal T}(S)$ as above. 
The fiber of the holomorphic cotangent bundle $T^*{\mathcal T}(S)$ over
$Z/\Gamma$ is identified with $H^0(Z,\, K^{\otimes 2}_Z)^\Gamma$. Therefore,
${\mathcal P}(S)$ is a holomorphic affine bundle over ${\mathcal T}(S)$ for the
holomorphic cotangent bundle $T^*{\mathcal T}(S)$.
\end{remark}

\section{Bers' Section}

Let us continue with the setting of the proof of Proposition \ref{prop1}:
we have a surface $Y$ and a finite group of diffeomorphisms $Y$, acting
on $Y$. An element $\gamma$ of $\Gamma$ induces a holomorphic automorphism of
$\mathcal T(Y)$ as well as of $\mathcal P(Y)$, which we denote by
\begin{equation}\label{im}
\gamma_{\mathcal T}\,:\,\mathcal T(Y)\,\longrightarrow\, \mathcal T(Y)~\ 
\text{ and }~\  \gamma_{\mathcal P}\,:\,\mathcal P(Y)\,\longrightarrow\,\mathcal P(Y)
\end{equation}
respectively. The mapping $f_Y$ of \eqref{fy} is $\Gamma$-equivariant,
which means that
\begin{equation}\label{fequivariant}
f_Y \circ \gamma_{\mathcal P} \,=\, \gamma_{\mathcal T}\circ f_Y\, .
\end{equation}

The above holomorphic mapping $\gamma_{\mathcal T}\,:\, \mathcal T(Y)\,\longrightarrow
\, \mathcal T(Y)$ induces in
a natural way a holomorphic self-map of the cotangent space
$$
d^*\gamma_{\mathcal T}\, :\, T^*\mathcal T(Y) \,\longrightarrow\, T^*\mathcal T(Y)\, .
$$
This space $T^*\mathcal T(Y)$, being the total space of the holomorphic cotangent
bundle of a complex manifold, has a natural holomorphic symplectic form,
which is known as the Liouville form. The Liouville symplectic form on
$T^*\mathcal T(Y)$ will be denoted by $\Omega_{\mathcal T}$. Since
$\gamma_{\mathcal T}$ is a biholomorphism of $\mathcal T(Y)$, it is easy to check
that the induced map $d^*\gamma_{\mathcal T}$ of $T^*\mathcal T(Y)$
preserves the form Liouville $\Omega_{\mathcal T}$.

On the other hand, the space $\mathcal P(Y)$ is mapped onto an open subset
of the smooth locus of the representation space
$$\text{Hom}(\pi_1(Y)\, , \text{PSL}(2,{\mathbb C}))/\text{PSL}(2,{\mathbb C})\, .$$
This map is a local biholomorphism.
Hence $\mathcal P(Y)$ has a natural holomorphic symplectic structure
\cite{AtBo}, \cite{Go}, \cite{AB}. We
will denote the holomorphic symplectic form on $\mathcal P(Y)$ by $\Omega_{\mathcal P}$.

\begin{lemma}\label{lem-spp}
The mapping $\gamma_{\mathcal P}$ in \eqref{im} preserves the symplectic
form $\Omega_{\mathcal P}$.
\end{lemma}

\begin{proof}
Let $Z$ be a compact connected oriented surface. Fix a base
point $z_0\,\in\, Z$. Let $G$ be a semisimple Lie
group. Consider $\text{Hom}(\pi_1(Z,z_0)\, , G)/G$ equipped with the
natural symplectic form $\Omega$ (see \cite{AtBo} and \cite{Go} for $\Omega$). Let
$$
\alpha\,:\, Z\, \longrightarrow\, Z
$$
be an orientation preserving diffeomorphism. The
isomorphism
$$
\alpha_*\,:\, \pi_1(Z,\, z_0)\, \longrightarrow\, \pi_1(Z,\,\alpha(z_0))
$$
induced by $\alpha$ produces a diffeomorphism
$$
\alpha'\,:\, \text{Hom}(\pi_1(Z,z_0)\, , G)/G\, \longrightarrow\,
\text{Hom}(\pi_1(Z,\,\alpha(z_0))\, , G)/G\, .
$$
By choosing a path from $z_0$ to $\alpha(z_0)$, the group
$\pi_1(Z,\,\alpha(z_0))$ is naturally identified with $\pi_1(Z,\,z_0)$ up
to an inner automorphism of $\pi_1(Z,\,z_0)$. Therefore, $\alpha'$ produces a
diffeomorphism
$$
\widetilde{\alpha}\,:\, \text{Hom}(\pi_1(Z,\,z_0)\, , G)/G\, \longrightarrow\,
\text{Hom}(\pi_1(Z,\,z_0)\, , G)/G\, .
$$
This diffeomorphism $\widetilde{\alpha}$ preserves $\Omega$. Indeed, this
follows immediately from the construction of $\Omega$ (see
\cite{AtBo} and \cite{Go}). We mentioned earlier
that $\Omega_{\mathcal P}$ coincides with the above symplectic form
$\Omega$ for $G\,=\, \text{PGL}(2,{\mathbb C})$. Therefore, we now conclude
that $\gamma_{\mathcal P}$ in \eqref{im} preserves $\Omega_{\mathcal P}$.
\end{proof}

The projection $f_Y$ in \eqref{fy} has a holomorphic section constructed by Bers
using the notion of simultaneous uniformization \cite{Bers}. We will denote that
section by $B$:
\begin{equation}\label{bers-section}
B\,:\,\mathcal T(Y) \,\longrightarrow\, \mathcal P(Y)\, .
\end{equation}
Let
\[
T_B\,:\, T^*{\mathcal T}(Y) \,\longrightarrow\,  \mathcal P(Y)
\]
be the holomorphic mapping that sends any $(Z\, ,\theta)\, \in\,
T^*{\mathcal T}(Y)$ to $B(Z)+\theta$; note that $\theta\,\in\,
H^0(Z,\, K^{\otimes 2}_Z)$ and the fiber of ${\mathcal P}(Y)$ over
$Z\,\in\, {\mathcal T}(Y)$ is an affine space for $H^0(Z,\, K^{\otimes 2}_Z)$,
so $B(Z)+\theta$ is also an element of the fiber of ${\mathcal P}(Y)$ over $Z$.
This map $T_B$ is clearly a biholomorphism
In \cite {ka}, Kawai proved that $T_B$ preserves
the symplectic structures of $\mathcal T(Y)$ and $\mathcal P(Y)$ in the
sense that
\begin{equation}\label{kt}
\frac{1}{\pi}\cdot T_B^*\Omega_{\mathcal P} \,=\, \Omega_{\mathcal T}\, .
\end{equation}

For an element $\gamma$ of $\Gamma$, we define a holomorphic mapping
\begin{equation}\label{bg}
B_\gamma\,:\,{\mathcal T}(Y) \,\longrightarrow\, {\mathcal P}(Y)\, ,~\
B_\gamma \,:=\, \gamma_{\mathcal P}\circ B\circ \gamma_{\mathcal T}^{-1}\, .
\end{equation}
We note that the following diagram is commutative
$$
\begin{matrix}
{\mathcal P}(Y) & \stackrel{\gamma_{\mathcal P}}{\longrightarrow} &
{\mathcal P}(Y)\\
~\Big\downarrow f_Y && ~\Big\downarrow f_Y\\
{\mathcal T}(Y) & \stackrel{\gamma_{\mathcal T}}{\longrightarrow} &
{\mathcal T}(Y)
\end{matrix}
$$
Therefore, $B_\gamma$ in \eqref{bg} is also a holomorphic section of
the projection $f_Y$.

Let
\begin{equation}\label{tbg}
T_{B_\gamma}\,:\, T^*{\mathcal T}(Y) \,\longrightarrow\,  \mathcal P(Y)
\end{equation}
be the biholomorphism that sends any $(Z\, ,\theta)\, \in\,
T^*{\mathcal T}(Y)$ to $B_\gamma(Z)+\theta$.

\begin{lemma}\label{le-sy}
For the above map $T_{B_\gamma}$ the following holds:
$$
T^*_{B_\gamma}\Omega_{\mathcal P} \,=\, \pi\cdot \Omega_{\mathcal T}\, .
$$
\end{lemma}

\begin{proof}
The following diagram of holomorphic maps is commutative
\begin{equation}\label{dc2}
\begin{matrix}
T^*{\mathcal T}(Y)& \stackrel{T_B}{\longrightarrow} &
{\mathcal P}(Y)\\
~\Big\downarrow d^*\gamma_{\mathcal T} && ~\Big\downarrow\gamma_{\mathcal P}\\
T^*{\mathcal T}(Y) & \stackrel{T_{B_\gamma}}{\longrightarrow} &
{\mathcal P}(Y)
\end{matrix}
\end{equation}
We noted earlier that $d^*\gamma_{\mathcal T}$ preserves the Liouville symplectic
form $\Omega_{\mathcal T}$ because $d^*\gamma_{\mathcal T}$ is induced by a
biholomorphism of ${\mathcal T}(Y)$. From Lemma \ref{lem-spp} we know that
$\gamma_{\mathcal P}$ preserves $\Omega_{\mathcal P}$. Therefore, in view of the
commutative diagram in \eqref{dc2}, from \eqref{kt} we conclude that
$T^*_{B_\gamma}\Omega_{\mathcal P} \,=\, \pi\cdot \Omega_{\mathcal T}$.
\end{proof}

Since the fibers of $f_Y$ are affine spaces for the fibers of the
holomorphic cotangent bundle of ${\mathcal T}(Y)$, we have
$$
B_\gamma - B\, \in\, H^0({\mathcal T}(Y),\, T^*{\mathcal T}(Y))\, ,
$$
in other words, $B_\gamma - B$ is a holomorphic one-form on ${\mathcal T}(Y)$.

\begin{lemma}\label{cor-s1}
The above holomorphic one-form $B_\gamma - B$ on ${\mathcal T}(Y)$ is closed.
\end{lemma}

\begin{proof}
Let
$$
\beta\, :\,  T^*{\mathcal T}(Y)\,\longrightarrow\, T^*{\mathcal T}(Y)
$$
be the holomorphic automorphism of the fiber bundle
\begin{equation}\label{p}
p\, :\, T^*{\mathcal T}(Y)\,\longrightarrow\, {\mathcal T}(Y)
\end{equation}
defined by $v\, \longmapsto\, v+(B_\gamma - B)(p(v))$.
Clearly, we have
$$
T_B\circ\beta\,=\, T_{B_\gamma}\, .
$$
Therefore, from \eqref{kt} and Lemma \ref{le-sy} it follows that
$$
\pi\cdot{\Omega}_T\,=\, (T_{B_\gamma})^*{\Omega}_P\,=\,
(T_B\circ\beta)^*{\Omega}_P\,=\, \beta^*(T_B)^*{\Omega}_P\,=\,\pi\cdot
\beta^*{\Omega}_T\, .
$$
Hence $\beta^* \Omega_{\mathcal T}-  \Omega_{\mathcal T}\,=\, 0$.
On the other hand, from the construction of $\Omega_{\mathcal T}$ it follows
immediately that
$$
\beta^* \Omega_{\mathcal T}- \Omega_{\mathcal T}\,=\, p^*d(B_\gamma - B)\, ,
$$
where $p$ is the projection in \eqref{p}. Combining these
two we conclude that the form $B_\gamma - B$ is closed.
\end{proof}

Since the fibers of $f_Y$ (see \eqref{fy}) are affine spaces for the fibers of $T^*
{\mathcal T}(Y)$, and $B_\gamma$ is a holomorphic section of $f_Y$, we conclude that
\begin{equation}\label{bp}
B'\,:=\, \frac{1}{\# \Gamma} \sum_{\gamma\in\Gamma} B_\gamma
\end{equation}
is a holomorphic section of $f_Y$, where $\# \Gamma$ is the order of
the group $\Gamma$.
Let
\begin{equation}\label{bp2}
T_{B'}\,:\, T^*{\mathcal T}(Y) \,\longrightarrow\,  \mathcal P(Y)
\end{equation}
be the holomorphic isomorphism that sends any $$(Z\, ,\theta)\, \in\,   
T^*{\mathcal T}(Y)$$ to $B'(Z)+\theta$.

\begin{proposition}\label{pr-sy}
For the above map $T_{B'}$, the following holds:
$$
T^*_{B'}\Omega_{\mathcal P} \,=\, \pi\cdot \Omega_{\mathcal T}\, .
$$
\end{proposition}

\begin{proof}
Let
\begin{equation}\label{o}
\omega\, :=\, \frac{1}{\# \Gamma} \sum_{\gamma\in\Gamma}(B_\gamma - B)
\end{equation}
be the holomorphic one-form on ${\mathcal T}(Y)$, where $B_\gamma - B$ is the
one-form in Lemma \ref{cor-s1}.
Let
$$
\beta'\, :\,  T^*{\mathcal T}(Y)\,\longrightarrow\, T^*{\mathcal T}(Y)
$$
be the holomorphic automorphism of the fiber bundle $T^*{\mathcal T}(Y)$
defined by $v\, \longmapsto\, v+\omega(p(v))$, where $p$
is the projection in \eqref{p}. Clearly, we have
\begin{equation}\label{ei}
T_B\circ\beta'\,=\, T_{B'}\, .
\end{equation}
As noted earlier, from the construction of $\Omega_{\mathcal T}$ it follows
immediately that
\begin{equation}\label{ea}
(\beta')^* \Omega_{\mathcal T}- \Omega_{\mathcal T}\,=\, p^*d\omega\,=\,
dp^*\omega\, .
\end{equation}
{}From Lemma \ref{cor-s1} we have
\begin{equation}\label{oo}
d\omega\,=\, \frac{1}{\# \Gamma} \sum_{\gamma\in\Gamma}d(B_\gamma - B)\,=\, 0\, .
\end{equation}
Hence $p^*d\omega\,=\, 0$. Consequently, from \eqref{ea} we have
$$
(\beta')^* \Omega_{\mathcal T}\,=\, \Omega_{\mathcal T}\, .
$$
Therefore, using \eqref{kt},
$$
(T_B\circ\beta')^*\Omega_{\mathcal P}\,=\,
(\beta')^*T_B^*\Omega_{\mathcal P}\,=\, (\beta')^*(\pi\cdot \Omega_{\mathcal T})
\,=\, \pi\cdot(\beta')^*\Omega_{\mathcal T}
\,=\, \pi\cdot \Omega_{\mathcal T}\, .
$$
Now from \eqref{ei} it follows that $T^*_{B'}\Omega_{\mathcal P}\,=\,
\pi\cdot \Omega_{\mathcal T}$.
\end{proof}

\section{Bers' section in the orbifold set-up}\label{se4}

As before,
$$
{\mathcal P}(Y)^\Gamma\, \subset\, {\mathcal P}(Y) ~ \ \text{ and }\ ~
{\mathcal T}(Y)^\Gamma\, \subset\, {\mathcal T}(Y)
$$
are the fixed point loci for the actions of $\Gamma$ on ${\mathcal P}(Y)$
and ${\mathcal T}(Y)$ respectively. Consider the projection $F_Y$ in \eqref{Fy}.
We will construct a holomorphic section of it.

{}From the construction of $B'$ in \eqref{bp} it follows immediately that
the action of the Galois group $\Gamma$ on ${\mathcal P}(Y)$
preserves the image $B'({\mathcal T}(Y))$ (the action leaves the subset
invariant, but not pointwise). The action of
$\Gamma$ on the Teichm\"uller space ${\mathcal T}(Y)$ produces an action of
$\Gamma$ on the cotangent bundle $T^*{\mathcal T}(Y)$. Since the projection $f_Y$
in \eqref{fy} is $\Gamma$--equivariant, and $B'({\mathcal T}(Y))$ is
preserved by the action of $\Gamma$, it follows that the biholomorphism
$T_{B'}$ in \eqref{bp2} is $\Gamma$--equivariant.

The image $B'({\mathcal T}(Y))$ being $\Gamma$--equivariant restricts to a
holomorphic section
\begin{equation}\label{wtb}
{\widetilde B}\, :\, {\mathcal T}(Y)^\Gamma \,\longrightarrow\,
{\mathcal P}(Y)^\Gamma
\end{equation}
of the projection $F_Y$ constructed in \eqref{Fy}. As noted before, for any
$$Z\,\in\, {\mathcal T}(Y)^\Gamma\, ,$$ the holomorphic cotangent space
$T^*_Z({\mathcal T}(Y)^\Gamma)$ coincides with the space of invariants
$$
H^0(Z,\, K^{\otimes 2}_Z)^\Gamma\, \subset\, H^0(Z,\, K^{\otimes 2}_Z)\, .
$$
We also recall that $F_Y$ is a holomorphic fiber bundle whose fiber
over any Riemann surface $Z\,\in\, {\mathcal T}(Y)^\Gamma$ is an affine
space for the vector space $H^0(Z,\, K^{\otimes 2}_Z)^\Gamma$. Therefore,
${\widetilde B}$ in \eqref{wtb} produces a biholomorphism
\begin{equation}\label{wtb2}
T_{\widetilde B}\, :\, T^*({\mathcal T}(Y)^\Gamma) \,\longrightarrow\,
{\mathcal P}(Y)^\Gamma
\end{equation}
that sends any $(Z\, ,\theta)\, \in\, T^*({\mathcal T}(Y)^\Gamma)$ to
${\widetilde B}(Z)+\theta\,\in\, {\mathcal P}(Y)^\Gamma$.

The symplectic form $\Omega_{\mathcal P}$ on $\mathcal P (Y)$ restricts to a
symplectic form on ${\mathcal P}(Y)^\Gamma$. This symplectic form on
${\mathcal P}(Y)^\Gamma$ will be denoted by $\Omega^\Gamma_{\mathcal P}$.
On the other hand, the cotangent bundle $T^*({\mathcal T}(Y)^\Gamma)$
is equipped with the Liouville symplectic form; this Liouville symplectic form
will be denoted by $\Omega^\Gamma_{\mathcal T}$.

\begin{proposition}\label{prop2}
For the biholomorphism $T_{\widetilde B}$ in \eqref{wtb2}, the following holds:
$$
T^*_{\widetilde B}\Omega^\Gamma_{\mathcal P}\,=\, \pi\cdot
\Omega^\Gamma_{\mathcal T}\, .
$$
\end{proposition}

\begin{proof}
Consider the action of $\Gamma$ on $T^*{\mathcal T}(Y)$ induced by the
action of $\Gamma$ on ${\mathcal T}(Y)$. It is easy to see that
the fixed-point set $(T^*{\mathcal T}(Y))^\Gamma\, \subset\, T^*{\mathcal T}(Y)$
is identified with $T^*({\mathcal T}(Y)^\Gamma)$. In particular, we have
$$
T^*({\mathcal T}(Y)^\Gamma)\, \subset\, T^*{\mathcal T}(Y)\, .
$$
The Liouville symplectic form $\Omega_{\mathcal T}$ on $T^*{\mathcal T}(Y)$
restricts to the Liouville symplectic form $\Omega^\Gamma_{\mathcal T}$ on
$T^*({\mathcal T}(Y)^\Gamma)$. The form $\Omega^\Gamma_{\mathcal P}$, by definition,
is the restriction of $\Omega_{\mathcal P}$. Also, The map $T_{\widetilde B}$
coincides with the restriction of $T_{B'}$ (constructed in \eqref{bp2}) to the
submanifold $T^*({\mathcal T}(Y)^\Gamma)$ of $T^*{\mathcal T}(Y)$.
Therefore, the proposition follows from Proposition \ref{pr-sy}.
\end{proof}

We recall that ${\mathcal P}(Y)^\Gamma$ and ${\mathcal T}(Y)^\Gamma$ are identified
with ${\mathcal P}(S)$ and ${\mathcal T}(S)$ respectively. Using these
identifications, the projection $F_Y$ in \eqref{Fy} coincides with the projection
$\widetilde{f}_S$ in \eqref{e4}.

The construction of the symplectic form on $\mathcal P (Y)$ extends to 
${\mathcal P}(S)$. Indeed, the symplectic form on the representation space of
a compact surface group constructed in \cite{Go}, \cite{AtBo} can be generalized
to the representation space of the fundamental group of a punctured surface
once we fix the monodromy around the punctures (see \cite{BG}). Let
$\Omega^S_{\mathcal P}$ denote the holomorphic symplectic form
on ${\mathcal P}(S)$. This form $\Omega^S_{\mathcal P}$ coincides with
$\Omega_{\mathcal P}$ using the above mentioned identification of
${\mathcal P}(Y)^\Gamma$ with ${\mathcal P}(S)$. The Liouville symplectic
form on $T^*{\mathcal T}(S)$ will be denoted by
$\Omega^S_{\mathcal T}$.

The section $\widetilde B$ (see \eqref{wtb}) of
the projection $F_Y$ produces a holomorphic section of
the projection $\widetilde{f}_S$ in \eqref{e4}. As done before, this
holomorphic section produces a biholomorphism
$$
T_{S,B}\, :\, T^*{\mathcal T}(S)\,\longrightarrow\, {\mathcal P}(S)\, .
$$
This map $T_{S,B}$ clearly coincides with $T_{\widetilde B}$ in \eqref{wtb2}
after identifying ${\mathcal P}(S)$ and $T^*{\mathcal T}(S)$ with 
${\mathcal P}(Y)^\Gamma$ and $T^*({\mathcal T}(Y)^\Gamma)$ respectively.

Therefore, Proposition \ref{prop2} gives the following:

\begin{theorem}\label{thm1}
For the above biholomorphism $T_{S,B}$, the following holds:
$$
T^*_{S,B}\Omega^S_{\mathcal P}\,=\, \pi\cdot
\Omega^S_{\mathcal T}\, .
$$
\end{theorem}

\section{Schottky and Earle uniformizations}

The projection $f_Y$ of \eqref{fy} admits a couple of other natural holomorphic
sections, apart from the one described in \eqref{bers-section}. One of these
sections is given by Earle in \cite{Earle}, which is a modification of the
simultaneous uniformization theorem. The other one is given by the uniformization by
Schottky groups. It is natural to ask whether the symplectic structures are preserved
by the biholomorphisms $T^*{\mathcal T}(S)\,\longrightarrow\, {\mathcal P}(S)$
constructed using these sections, i.e., whether the analogue of Theorem \ref{thm1}
holds. Our aim in this final section is to address this question.

In \cite{Earle}, Earle constructed a holomorphic section
\[
e\,:\, \mathcal T(Y) \,\longrightarrow\, \mathcal P(Y)\,.
\]
The construction of $e$, which follows closely the approach of the simultaneous
uniformization theorem of Bers, is done using a marking on $Y$ and an involution
of the fundamental group of $Y$ induced by an orientation reversing diffeomorphism
of $Y$.
The construction of $e$ follows a modification of the simultaneous uniformization
theorem. In a sense this section is intrinsic, since it does not require fixing a base
point of $\mathcal T(Y)$, unlike in the construction of Bers' section. As said above,
the section $e$ is holomorphic.

For each element $\gamma\,\in\,\Gamma$, we get another section $e_\gamma$ defined by
$$e_\gamma \,=\,\gamma_{\mathcal P}\circ e \circ \gamma_{\mathcal T}^{-1}$$ (just as
done in \eqref{bg}). Let $\rho_\gamma$ denote the one-form on $\mathcal T(Y)$ defined by
\begin{equation}
\rho_\gamma \,:=\, e_\gamma - B_\gamma
\end{equation}

\begin{lemma}\label{lemma51}
The form $\rho_\gamma$ is closed.
\end{lemma}

\begin{proof}
Let $\phi$ be the $C^\infty$ section of the projection the projection $f_Y$ (see
\eqref{fy}) given by the Fuchsian uniformization. This section is not
holomorphic. We define
\[
\alpha_\gamma \,:= \, e_\gamma - \phi \ ~\text{ and }\ ~
\beta_\gamma \,:=\, B_\gamma - \phi\, .
\]
Since $e_\gamma$ and $B_\gamma$ are holomorphic sections, Theorem 9.2 of
\cite[p. 355]{Mcm} applies, and from it we conclude that
\[
d\alpha_\gamma \,= \, d\beta_\gamma\, .
\]
Therefore, $d\rho_\gamma \,=\, d(e_\gamma-B_\gamma)\,=\,
d\alpha_\gamma - d\beta_\gamma\,=\, 0$.
\end{proof}

Let
\[
T^\gamma_E\,:\, T^*{\mathcal T}(Y) \,\longrightarrow\,  \mathcal P(Y)
\]
be the biholomorphism that sends any $(Z\, ,\theta)\, \in\,
T^*{\mathcal T}(Y)$ to $e_\gamma(Z)+\theta\,\in\, \mathcal P(Y)$. Let
$$
A_{\rho_\gamma}\, :\, T^*{\mathcal T}(Y) \,\longrightarrow\,
T^*{\mathcal T}(Y)
$$
be the holomorphic automorphism defined by $v\, \longmapsto\,
v+\rho_\gamma(p(v))$, where $p$ is the projection in \eqref{p}. Clearly,
\begin{equation}\label{61}
T^\gamma_E\,=\, T_{B_\gamma}\circ A_{\rho_\gamma}\, ,
\end{equation}
where $T_{B_\gamma}$ is constructed in \eqref{tbg}. Now, using \eqref{kt}
and Lemma \ref{lemma51},
$$
(T_{B_\gamma}\circ A_{\rho_\gamma})^*\Omega_{\mathcal P}\,=\,
(A_{\rho_\gamma})^*(T_{B_\gamma})^*\Omega_{\mathcal P}
$$
$$
=\,
\pi\cdot (A_{\rho_\gamma})^*\Omega_{\mathcal T}\,=\,
\pi\cdot (\Omega_{\mathcal T}+d\rho_\gamma)\,=\,
\pi\cdot\Omega_{\mathcal T}\, .
$$
Therefore, from \eqref{61} we have
\begin{equation}\label{zz}
(T^\gamma_E)^*\Omega_{\mathcal P} \,=\, \pi\cdot\Omega_{\mathcal T}\,.
\end{equation}

We now average these sections, and define
\begin{equation}\label{ep}
e'\,:=\, \frac{1}{\# \Gamma} \sum_{\gamma\in\Gamma}e_\gamma\, ,
\end{equation}
which is a holomorphic section of $f_Y$. Let
\[
T_{e'} \,:\, T^*\mathcal T(Y) \,\longrightarrow\, \mathcal P(Y)
\]
be the biholomorphism that sends any $(Z\, ,\theta)$ to $e'(Z)+\theta$
(as in \eqref{bp2}).

\begin{proposition}
For the above map $T_{e'}$, the following holds:
\[
T^*_{e'}\Omega_{\mathcal P} \,=\, \pi\cdot \Omega_{\mathcal T}\,.
\]
\end{proposition}

\begin{proof}
As in Proposition \ref{pr-sy}, we define a holomorphic one-form on $T(Y)$
by
\[
\mu\, :=\, \frac{1}{\# \Gamma} \sum_{\gamma\in\Gamma}(e_\gamma - e)\, .
\]
In view of \eqref{zz}, it suffices  to show that $\mu$ is closed (see
the proof of Proposition \ref{pr-sy}).

Consider $\omega$ constructed in \eqref{o}. We observe that
$$
\mu-\omega\,=\, \frac{1}{\# \Gamma} \sum_{\gamma\in\Gamma} (e_\gamma - e - B_\gamma + B)
\,=\,\frac{1}{\# \Gamma} \sum_{\gamma\in\Gamma} ((e_\gamma-B_\gamma) - (e-B))
$$
$$
\,=\,
\frac{1}{\# \Gamma} \sum_{\gamma\in\Gamma\setminus{e^0}}(e_\gamma-B_\gamma)\, ,
$$
where $e^0$ is the identity element of $\Gamma$. Hence from Lemma \ref{lemma51}
it follows that $d(\mu-\omega)\,=\, 0$. Now from
\eqref{oo} we conclude that $d\mu\,=\,0$.
\end{proof}

The constructions of Section \ref{se4} can be done with the section $e'$
in \eqref{ep} instead of $B'$, to obtain a section
$$
\widetilde{e}\,:\, \mathcal T(Y)^\Gamma \,\longrightarrow\, \mathcal P(Y)^\Gamma
$$
of the projection $F_Y$ in \eqref{Fy}. Just as before, $\widetilde{e}$
produces a biholomorphism
\[
T_{S,e}\,:\, T^*\mathcal T(S) \,\longrightarrow\, \mathcal P(S)\, .
\]

Now just as in Theorem \ref{thm1}, we have:

\begin{theorem}\label{thm2}
For biholomorphic mapping $T_{S,e}$,
\[
T_{S,e}^*\Omega^S_{\mathcal P} \,=\, \pi\cdot \Omega^S_{\mathcal T}\, .
\]
\end{theorem}

We finish this section with the observation that all the above constructions
carry to the case of the Schottky section. This section also satisfies
McMullen's theorem (Theorem 9.2 of \cite[p. 355]{Mcm}) which says the
following: Let $X\,=\,\Omega/\Gamma$ be the quotient Riemann surface
for a finitely generated Kleinian group $\Gamma$, and
let $\mu,\, \nu \,\in\, M(X)$ be a pair of
sufficiently smooth Beltrami differentials.
Then
$$
\int_X \phi_\mu \, \nu = \int_X \phi_\nu \, \mu\, ,
$$
where $\phi_\mu,\, \phi_\nu \,\in\, L^1(X,dz^2)$ give
the projective distortions of $\mu$ and $\nu$.

Therefore, Theorem \ref{thm2} remains valid for the biholomorphism $$T^*
{\mathcal T}(S) \,\longrightarrow\, \mathcal P(S)$$ given by the Schottky section.

\section*{Acknowledgements}

We thank the referee for helpful comments. The second-named
author is partially supported by a J. C. Bose Fellowship.

\end{document}